\documentclass[12pt]{amsart}

\usepackage{amssymb}
\usepackage{amsthm}
\usepackage{amsmath}
\usepackage{amsfonts}
\usepackage{latexsym} 
\usepackage{eucal}
\usepackage{indentfirst} 
\usepackage{mathrsfs}

\newtheorem{theorem}{Theorem}[section]

\newtheorem{corollary}[theorem]{Corollory}

\newcommand{\be}{\begin{equation}}
\newcommand{\ee}{\end{equation}}

\begin{document}

\def \cd {, \ldots ,}
\def \lf{\| f \|}
\def \bs{\setminus}
\def \ep{\varepsilon}
\def \sig{\Sigma}
\def \si {\sigma}
\def \gam {\gamma}
\def \cinf {C^\infty}
\def \cid {C^\infty (\partial D)}

\def \hx {\widehat X}
\def \hxx {\widehat X \bs X}
\def \hy {\widehat Y}
\def \hyy {\widehat Y \bs Y}
\def \hj {\widehat J}
\def \hjj {\widehat J \bs J}

\def \ha {\widehat A}
\def \hb {\widehat B}
\def \hk {\widehat K}

\def \Z {\mathbb Z}
\def \C {\mathbb C}
\def \CN {\mathbb C^N}
\def \CNO {\mathbb C^{N+1}}

\def \od {\overline D}
\def \ol {\overline L}
\def \oj {\overline J}

\def \cn {c_1, \ldots , c_n}
\def \zx { z \in \widehat X }
\def \zxx {\{ z \in \widehat X \} }
\def \siv {\sum^\infty _{j=1}}
\def \sivv {\sum^\infty _{j=n+1}}
\def \smv {\sum^m_{j=n+1}}
\def \snv {\sum^n_{j=j_0+1}}
\def \bjkz {B_{jk} (z_0)}
\def \epp {\varepsilon^2_n/2}
\def \epn {\varepsilon_{n+1}}
\def \nn { {n+1} }
\def \nnn {{2N}}
\def \ff {F_{z_{0}}}
\def \fk {f_{k_{1}}}
\def \fkv {f_{k_{v}}}
\def \tf {\tilde{f}}

\def \ma {\mathfrak{M}_A}  
\def \mb {\mathfrak{M}_B}
\def \muc {\mathfrak{M}_{\mathscr U}}
\def \uc {\mathscr U}
\def \xy {(x,y)}
\def \fg {f \otimes g}
\def \vp {\varphi}
\def \pl {\partial L}
\def \ad {A(D)}
\def \aid {A^\infty (D)}
\def \ot {\otimes}
\def \rbl {R_b (L)}
\def\Int {{\rm Int}}

\def \sE {\mathscr E}
\def\endhat{\widehat{\phantom j}}

\subjclass[2000]{32E20}
\title[Cantor set]{A Cantor set whose polynomial hull \\ contains no analytic discs}
\author{Alexander J. Izzo}
\address{Department of Mathematics and Statistics, Bowling Green State University, Bowling Green, OH 43403}
\email{aizzo@bgsu.edu}
\author{Norman Levenberg}
\address{Department of Mathematics, Indiana University, Bloomington, IN 47405}
\email{nlevenbe@indiana.edu}

\begin{abstract}
\vskip 24pt
A generalization of a result of Wermer concerning the existence of polynomial hulls without analytic discs is presented.  As a consequence it is shown that there exists a Cantor set $X$ in $\C^3$ whose polynomial hull is strictly larger than $X$ but contains no analytic discs.
\end{abstract}

\maketitle

 \section{Introduction}
It was once conjectured that whenever the polynomial hull $\hx$ of a compact set $X$ in $\CN$ is strictly larger than $X$, the complementary set $\hxx$ must contain an analytic disc. This conjecture was disproved by Gabriel Stolzenberg \cite{Stol1}.  However, when $X$ is a smooth one-dimensional manifold, the set $\hxx$, if nonempty, is a one-dimensional analytic variety as was also shown by Stolzenberg \cite{Stol2} (strengthening earlier results of several mathematicians).
In contrast, recent work of the first author, H\aa kan Samuelsson Kalm, and Erlend Forn{\ae}ss Wold \cite{ISW} and the first author and Lee Stout \cite{IS} shows that every smooth manifold of dimension strictly greater than one smoothly embeds in some $\CN$ as a subspace $X$ such that $\hxx$ is nonempty but contains no analytic discs.

In response to a talk on the above results given by the first author of the present paper, Hari Bercovici raised the question of whether a {\it nonsmooth} one-dimensional manifold can have polynomial hull containing no analytic discs. This question was the motivation for the present paper.  The authors would like to thank Bercovici for his stimulating question.

In the paper \cite{Izzo} of the first author, a construction is given that yields a Cantor set in $\C^3$ such that $\hxx$ is nonempty but contains no analytic discs, and the result is used there to answer Bercovici's question affirmatively by showing that, in fact, 
every uncountable, compact subspace of a Euclidean space can be embedded in some $\CN$ so as to have polynomial hull containing no analytic discs. 
In the present paper, a different construction of a Cantor set in $\C^3$ such that $\hxx$ is nonempty but contains no analytic discs is presented.  The Cantor sets given by the construction in \cite{Izzo} are different from the ones given by the construction presented here, and they have different properties.  For instance, the Cantor sets $X$ in \cite{Izzo} have the stronger property that the uniform algebra $P(X)$ has a dense set of invertible elements.  This implies that the polynomial and rational hulls of $X$ coincide.  In contrast, the construction given here does not yield density of invertible elements, and as shown in the first author's paper \cite{Izzo2}, it can be used to obtain rationally convex examples.

It will be convenient to say that the polynomial hull of a set $X \subset \CN$ is {\it nontrivial} if the set $\hxx$ is nonempty.

The Cantor sets of the present paper will be obtained by generalizing a result of John Wermer \cite{Wermer1982} inspired by an idea of Brian Cole \cite{Cole}.  Specifically, Wermer produced a compact subset $Y$ of $\C^2$ such that, letting $\pi : \C^2 \to \C$ be projection onto the first coordinate, $\pi (Y)$ is the unit circle, $\pi (\widehat Y)$ is the closed unit disc, and $\widehat Y$ contains no analytic discs. 
We generalize this by replacing the unit circle by an arbitrary compact set $X \subset \CN$ with nontrivial polynomial hull and replacing the closed unit disc by~$\hx$.

\begin {theorem} \label{Wermergen} 
Let $X \subset \CN$ be a compact set whose polynomial hull is nontrivial.  Then there exists a compact set $Y \subset \CNO$ such that, letting $\pi$ denote the restriction to $\hy$ of the projection $\CNO \to \CN$ onto the first $N$ coordinates, the following conditions hold:
\item {\rm (i)} $\pi (Y) = X$
\item {\rm (ii)} $\pi (\hyy) = \hxx$
\item {\rm (iii)} $\hy$ contains no analytic discs
\item {\rm (iv)} each fiber $\pi^{-1} (z)$ for $z \in \hx$ is totally disconnected.
\end {theorem}

By applying this theorem with $X$ taken to be a Cantor set in $\C^2$ with nontrivial polynomial hull, we will obtain the promised Cantor set.
The first example of a Cantor set with nontrivial polynomial hull was constructed by Walter Rudin \cite{Rudin-Cantor} using a modification of an argument of Wermer  \cite{Wermer-Cantor} who produced the first example of an {\it arc\/} with nontrivial polynomial hull.  Later, examples of Cantor sets whose polynomial hulls even have interior in $\C^2$ were given by Vitushkin~\cite{Vit}, J\"oricke~\cite{J1, J2}, and Henkin~\cite{Henkin}.

\begin {corollary} \label{Cantorset} 
There exists a Cantor set in $\C^3$ whose polynomial hull is nontrivial but contains no analytic discs.
\end {corollary}

In an earlier draft of the present paper (see \cite{IzzoL}), Theorem~\ref{Wermergen}, but with $\C^{2N}$ in place of $\CNO$, was proven by generalizing the argument given by Wermer for the special case when $X$ is the unit circle in the plane given in \cite{Wermer1982} (and also presented in \cite[pp.~211--218]{Alex-Wermer}).
The proof of Theorem~\ref{Wermergen} presented below is based on the suggestions of a referee.  The authors thank the referee for suggesting this approach in which the set $Y$ is obtained from a set $\sE$ constructed by 
Tobias Harz, Nikolay Shcherbina, and Giuseppe Tomassini \cite{HST}.  The set $\sE$ is constructed in \cite{HST} by a generalization of Wermer's construction the technical details of which are considerably more complicated than those in \cite{IzzoL}.   It seems likely that a construction alongs the lines of the one given in \cite{IzzoL} could be used to give a simpler proof of a statement along the lines of \cite[Theorem~1.1]{HST} in the case when $n$ is an even integer.

As a step toward proving Theorem~\ref{Wermergen}, we will prove the following general  result which may have further applications.

\begin{theorem}\label{refereeIzzo}
Let $\Sigma\subset \CNO$ be a set, and let $\tau:\Sigma\rightarrow \CN$ denote the restriction of the projection $\CNO\rightarrow \CN$ onto the first $N$ coordinates.  Suppose that 
\item {\rm (i)} $\Sigma$ is closed in $\CNO$
\item {\rm (ii)} each point of $\CN$ has a neighborhood $U$ such that the set 
$\{z\in \C: (\lambda, z)\in \Sigma {\rm \ for\ some\ } \lambda \in U\}$ is bounded
\item {\rm (iii)} $\tau$ is surjective
\item {\rm (iv)} $\CNO\bs \Sigma$ is pseudoconvex.\hfil\break
Then for every compact set $X\subset \CN$, we have
$\tau^{-1}(\hx)\subset [\tau^{-1}(X)]\endhat$.
\end{theorem}

Note that in the context of Theorem~\ref{refereeIzzo}, the set $\tau^{-1}(\hx)$ need not be polynomially convex, and thus the \emph{equality} $\tau^{-1}(\hx)= [\tau^{-1}(X)]\endhat$ does not in general hold.  For instance if $\Sigma=\{ (z_1, z_2)\in \C^2: |z_1|^2 + |z_2|^2 \geq 1 {\rm\ and\ } |z_2|\leq 2\}$, then $\tau^{-1}(\{ z\in \C: |z|\leq 1\})$ contains the unit sphere in $\C^2$ but does not contain its polynomial hull, the closed unit ball.

Before concluding this introduction, we make explicit the definitions of some terms already used above.  For a compact set $X$ in $\CN$, the \emph{polynomial hull} $\hx$ of $X$ is defined by
$$\hx=\{z\in\CN:|p(z)|\leq \max_{x\in X}|p(x)|\
\mbox{\rm{for\, all\, polynomials}}\, p\}.$$
The set $X$ is said to be \emph{polynomially convex} if $\hx = X$.
By an \emph{analytic disc} in $\CN$, we mean an injective holomorphic map $\sigma: \{ z\in \C: |z|< 1\}\rightarrow\CN$.
By the statement that a subset $S$ of $\CN$ contains no analytic discs, we mean that there is no analytic disc in $\CN$ whose image is contained in $S$.

The research in this paper was begun while the first author was a visitor at Indiana University.  He would like to thank the Department of Mathematics for its hospitality.

\section {The Proofs}

\begin{proof}[Proof of Theorem~\ref{refereeIzzo}]
Let $K$ be the mapping from $\CN$ into the set of nonempty compact subsets of $\C$ whose graph is $\Sigma$, i.e., such that
\[
K(\lambda)=\{z\in \C: (\lambda, z)\in \Sigma\}.
\]
Note that the requirement that $K(\lambda)$ is nonempty for each $\lambda$ is fulfilled by condition (iii).  We leave it as an exercise to verify that the combination of conditions (i) and (ii) is equivalent to the statement that $K$ is upper semicontinuous in the sense that for every $\lambda_0\in \CN$ and every open set $V$ of $\C$ containing $K(\lambda_0)$, there exists a neighborhood $U$ of $\lambda_0$ such that $K(\lambda)\subset V$ for every $\lambda\in U$.
Therefore, condition (iv) gives, by S\l odkowski's theorem on the equivalence of conditions for a set-valued function of one complex variable to be an analytic multifunction 
\cite[Theorem~3.2]{Slod} (or\vadjust{\kern 2pt} see \cite[Theorem~7.1.10]{Au}), that for every plurisubharmonic function $\phi$ on $\CNO$, the function $\Phi$ on $\CN$ defined by 
\[
\Phi(\lambda)=\max\{ \phi(\lambda, z): z\in K(\lambda)\}
\]
is subharmonic on each complex line in $\CN$ and hence is plurisubharmonic on $\CN$.

Now let $w_0\in \tau^{-1}(\hx)$ be arbitrary, and let $p$ be an arbitrary polynomial on $\CNO$.  Since $|p|$ is plurisubharmonic, the conclusion of the preceding paragraph gives that the function $\psi$ on $\CN$ defined by
\[
\psi(\lambda)=\max\{\, |p(\lambda, z)|:z\in K(\lambda)\,\}
\]
is plurisubharmonic.
Obviously
\begin{equation}\label{1}
|p(w_0)|\leq \psi\bigl(\tau(w_0)\bigr).
\end{equation}
Since $\tau(w_0)$ is in $\hx$, the well-known equality of holomorphic hull and plurisubharmonic hull (see for instance \cite[Theorem~4.3.4]{Hor} or 
\cite[Theorem~VI.1.18]{Range}) gives that
\begin{equation}\label{2}
\psi(\tau\bigl(w_0)\bigr)\leq \max\{\psi(\lambda):\lambda\in X\}.
\end{equation}
Combining (\ref{1}) and (\ref{2}), we have
\[
|p(w_0)|\leq \max\{\psi(\lambda):\lambda\in X\}.
\]
But the definition of $\psi$ gives that
\[
\max\{\psi(\lambda):\lambda\in X\}=\max\{|p(w)|:w\in \tau^{-1}(X)\}.
\]
Thus
\[
|p(w_0)|\leq \max\{|p(w)|: w\in \tau^{-1}(X)\},
\]
so $w_0$ is in $[\tau^{-1}(X)]\endhat$.
We conclude that $\tau^{-1}(\hx)\subset [\tau^{-1}(X)]\endhat$.
\end{proof}

\begin {proof}[Proof of Theorem~\ref{Wermergen}]  
Let $\sE \subset \CNO$ be the set given by \cite[Theorem~1.1]{HST} and let $\tau:\sE\rightarrow \CN$ denote the restriction of the projection $\CNO\rightarrow \CN$ onto the first $N$ coordinates.  The set $\sE$ has the following properties:
\item{\rm(a)} $\sE$ is closed in $\CNO$
\item{\rm(b)} each point of $\CN$ has a neighborhood $U$ such that the set 
$\{z\in \C: (\lambda, z)\in \sE {\rm \ for\ some\ } \lambda \in U\}$ is bounded
\item{\rm(c)} $\tau$ is surjective
\item{\rm(d)} $\tau^{-1}(z)$ is totally disconnnected for each $z\in \CN$
\item{\rm(e)} $\CNO\bs \sE$ is pseudoconvex
\item{\rm(f)} for every $R>0$, the intersection of $\sE$ with the closed ball of radius $R$ centered at the origin in $\CNO$ is polynomially convex
\item{\rm(g)} $\sE$ contains no analytic discs.\hfil\break
Properties (a), (e), (f), and (g) are given in \cite[Theorem~1.1]{HST}.  The other properties are not given in the statement of the theorem; the reader is invited to examine the proof of \cite[Theorem~1.1]{HST} to verify that the set constructed there has these properties.

Let $Y=\tau^{-1}(X)$.  Then $Y$ is compact by conditions (a) and (b), and (i) of the theorem holds on account of condition (c).  It is easily shown that condition (f) yields that $\tau^{-1}(\hx)$ is polynomially convex.  Thus $\hy\subset \tau^{-1}(\hx)\subset \sE$, so (iii) and (iv) of the theorem follow from conditions (g) and (d), respectively.  In view of conditions (a), (b), (c), and (e), Theorem~\ref{refereeIzzo} 
gives\vadjust{\kern 3pt} that $\tau^{-1}(\hx)\subset [\tau^{-1}(X)]\endhat = \hy,$
so $\hy=\tau^{-1}(\hx)$.  Therefore, condition (c) gives  that $\pi(\hy)=\hx$, and this together with the equality $Y=\tau^{-1}(X)$ yields $\pi (\hyy) = \hxx$, i.e., (ii) of the theorem holds.
\end{proof}

\begin {proof} [Proof of Corollary~\ref{Cantorset}:]
As mentioned in the introduction, there are several examples in the literature of Cantor sets in $\C^2$ having nontrivial polynomial hull.  Let $X$ be one of these Cantor sets, and let $Y$ be the set in $\C^3$ given by then applying Theorem~\ref{Wermergen}.  
Let $J$ be the largest perfect subset of $Y$.  (Recall that a subset of a space is called \emph{perfect} if it is closed and has no isolated points.  Every space contains a unique largest perfect subset (which can be empty), namely the closure of the union of all perfect subsets of the space.)
By \cite[Lemma~4.2]{Izzo}, $\hjj \supset \hyy$, so condition (ii) of Theorem~\ref{Wermergen} gives that $\hj$ is nontrivial, and since $\hj \subset \hy$, condition (iii) gives that $\hj$ contains no analytic discs.  
It follows from conditions (i) and (iv) of Theorem~\ref{Wermergen} that  $J$ is a Cantor set by the well-known characterization of Cantor sets as the compact, totally disconnected, metrizable spaces without isolated points.  
\end{proof}

\end{document}